\documentclass[11pt]{article}

\usepackage{graphicx}
\usepackage{amsfonts,amssymb,amsmath}
\usepackage{bm}          
\usepackage{geometry}
\usepackage[all]{xy}
\usepackage[usenames,dvipsnames]{xcolor}
\usepackage{hyperref}

\newcommand{\corurl}{red}
\newcommand{\corcite}{ForestGreen}
\newcommand{\corlink}{blue}

\hypersetup{linktocpage,colorlinks,urlcolor=\corurl,citecolor=\corcite,linkcolor=\corlink,
pdftitle={About the reducibility of the variety of complex Leibniz algebras},
pdfauthor={J.M. Ancochea Bermudez and J. Margalef--Bentabol}}

\newtheorem{theorem}{Theorem}

\newtheorem{proposition}{Proposition}
\newtheorem{corollary}{Corollary}
\newtheorem{definition}{Definition}
\newtheorem{remark}{Remark}

\newenvironment{proof}[1][Proof]{\textbf{#1.} }{\
  \rule{0.5em}{0.5em}}

\newcommand*{\leib}{{\rm Leib}}
\newcommand*{\leibn}{{\rm LeibN}}
\newcommand*{\lie}{{\rm L}}
\newcommand*{\varl}{\mathcal{L}}
\newcommand*{\ncom}{\mathbb{C}}
\newcommand*{\algl}{\mathfrak{l}}
\newcommand*{\suitc}{\mathcal{C}}

\title{About the reducibility of the variety of complex Leibniz algebras}

\author{J.M. Ancochea Berm\'udez and J. Margalef--Bentabol}

\begin{document}

\centerline{\textsf{\textbf{\huge{About the reducibility of the variety}}}}\mbox{}\vspace*{0.5ex}
\centerline{\textsf{\textbf{\huge{of complex Leibniz algebras}}}}\mbox{}\vspace*{0.5ex}

\begin{center}
  \begin{tabular}{ccccc}
    \textsf{J.M. Ancochea Berm\'udez} & \quad & \textsf{J. Margalef--Bentabol} & \quad & \textsf{J. S\'anchez Hern\'andez}\\[0.1ex]
    \href{mailto:ancochea@mat.ucm.es}{ancochea@mat.ucm.es} & \quad & \href{mailto:juanmargalef@ucm.es}{juanmargalef@ucm.es} & \quad & \href{mailto:jnsanchez@mat.ucm.es}{jnsanchez@mat.ucm.es}\\ \mbox{}
  \end{tabular}

  Dpto. Geometr\'{\i}a y Topolog\'{\i}a,\\ Facultad CC. Matem\'aticas UCM\\Plaza de Ciencias 3, E-28040 Madrid
\end{center}

\begin{abstract}
  \noindent In this paper, using the notions of perturbation and contraction of Lie and Leibniz algebras, we show that the algebraic varieties of Leibniz and nilpotent Leibniz algebras of dimension greater than 2 are reducible.
\end{abstract}

\medskip


{\bfseries\slshape AMS Classification Numbers:\/} 17A32

\section{Definition and preliminary properties}
The aim of this work is to prove the reducibility of the Leibniz and nilpotent Leibniz algebraic varieties, that we will denote $\leib^n$ and $\leibn^n$ respectively. First we will classify the 3-dimensional nilpotent Leibniz algebras over the complex field. Then, using the internal set theory of Nelson \cite{Ne}, we will introduce what a perturbation of a Leibniz algebra is. Such notion allows us to determine the open components of the variety $\leibn^3$, it will turn out to have two of them. The other algebras are obtained as the limit by contraction of the rigid algebra or the family of rigid algebras defining the open components. Moreover, we characterize the rigid Lie algebras over the variety $\lie^n$ of Lie algebras that are rigid over $\leib^n$.

\begin{definition}
  A \textbf{Leibniz algebra law} $\mu$ over $\mathbb{C}$ is a bilinear map $\mu: \mathbb{C}^{n}\times \mathbb{C}^{n}\rightarrow \mathbb{C}^n$ satisfying
\begin{equation}
\mu\Big(x,\mu(y,z)\Big)=\mu\Big(\mu(x,y),z\Big)-\mu\Big(\mu(x,z),y\Big). \label{Ldef}
\end{equation}
We call \textbf{Leibniz algebra} to any pair $(\mathbb{C}^n,\mu)$ where  $\mu$ is a Leibniz algebra law.
\end{definition}
The previous equation is known as the Leibniz identity. From now on, the law will be identified with its algebra and the non written products will be supposed to be zero.

Notice that if $\mu$ is anticommutative $\mu(x,y)=-\mu(y,x)$, the Leibniz identity is equivalent to the Jacobi one
\begin{equation}
  \mu\Big(x,\mu(y,z)\Big)+\mu\Big(y,\mu(z,x)\Big)+\mu\Big(z,\mu(x,y)\Big)=0,
\end{equation}
as~\eqref{Ldef} is obtained by placing the element $x$ on the first place at every element of the Jacobi identity.

Let $\algl=(\mathbb{C}^n,\mu)$ be a Leibniz algebra, we define the \textbf{right-decreasing central sequence} as
\[\suitc^1(\algl)=\algl \qquad \suitc^2(\algl)=\mu(\algl,\algl)\qquad \cdots\qquad \suitc^{k+1}(\algl)=\mu(\suitc^k(\algl),\algl)\qquad \cdots\]

\begin{definition}
A Leibniz algebra $\algl$ is \textbf{nilpotent} if there exists some $k\in\mathbb{N}$ such that $\suitc^k(\algl)=\{0\}$.
\end{definition}

\medskip

For a given nilpotent Leibniz algebra $\algl=(\ncom^n,\mu)$, we define for every $x\in\ncom^n$ the endomorphism $R_x:\ncom^n\rightarrow\ncom^n$ as
\[R_x(y)=\mu(y,x),\quad \forall y\in\ncom^n.\]
It is easy to check that $R_x$ is a nilpotent endomorphism, then for any $x\in\algl\setminus\mathcal{C}^2(\algl)$, we write $s_\mu(x)=(s_1(x),\ldots,s_k(x))$ the decreasing sequence $s_1\geq s_2\geq\ldots\geq s_k$ of dimensions of the Jordan blocks of the nilpotent operator $R_x$. We may now order lexicographically $s_\mu(x)$ for all $x\in\algl\setminus\mathcal{C}^2(\algl)$ and denote $s(\mu)$ its maximum which is, up to isomorphism, an invariant of the isomorphism class of the algebra $\algl$. We call it \textbf{characteristic sequence} of $\algl$. If $x\in \algl\setminus\suitc^2(\algl)$ satisfies  $s_\mu(x)=s(\mu)$, we say that $x$ is a \textbf{characteristic vector} of $\algl$.

We will denote $\leib^n$ the set of all Leibniz algebras over $\ncom^n$ and $\leibn^n$ the set of nilpotent Leibniz algebras over $\ncom^n$.

Notice that we can identify any Leibniz algebra $\mu$ with its structure constants over a fixed base. Given $\{e_1,\ldots,e_n\}$ a base of $\ncom^n$, from the identity \eqref{Ldef} we have that the coordinates defined by $\mu(e_i,e_j)=a_{ij}^ke_k$ are the solution of
\begin{equation}
  a_{jk}^la_{il}^m-a_{ij}^la_{lk}^m+a_{ik}^la_{lj}^m=0,\ \ \ 1\leq i,j,k,m\leq n
\end{equation}
As the nilpotent conditions are also polynomials, $\leib^n$ and $\leibn^n$ can be endowed with an algebraic structure over $\ncom^{n^3}$.

\section{Classification of the nilpotent Leibniz algebra of dimension 3}\label{seccion classification}

Let $\algl=(\ncom^3,\mu)$ be a nilpotent Leibniz algebra. According to the previous section, the possible characteristic sequences of $\algl$ are $\suitc(\algl)\in\{(3),(2,1),(1,1,1)\}$.
\begin{enumerate}
  \item If $s(\algl)=(3)$, there exists a characteristic vector $e_1$ and a base $\{e_1,e_2,e_3\}$ such that
    \begin{align*}
      \mu(e_1,e_1)&=e_2,\\
      \mu(e_2,e_1)&=e_3.
    \end{align*}
    As $\mu(x,e_2)=\mu(x,\mu(e_1,e_1))=\mu(\mu(x,e_1),e_1)-\mu(\mu(x,e_1),e_1)=0$, we have that $R_{e_2}= 0$. The Leibniz identity for $(e_1,e_2,e_1)$, $(e_2,e_2,e_1)$ and $(e_3,e_2,e_1)$ shows that $\mu(e_1,e_3)=\mu(e_2,e_3)=\mu(e_3,e_3)=0$. In this case thus, there only exists (up to isomorphism) one nilpotent Leibniz algebra $\mu_1$ of maximal characteristic sequence, which is given by
    \begin{align*}
      \mu_1(e_1,e_1)&=e_{2},\\
      \mu_1(e_2,e_1)&=e_3.
    \end{align*}
  \item If $s(\algl)=(2,1)$, we have two possibilities
    \begin{enumerate}
      \item There exists a characteristic vector $e_1$ such that $\mu(e_1,e_1)\neq 0$.
      \item For every characteristic vector $x$, we have $\mu(x,x)=0$.
    \end{enumerate}
    For the $(a)$ case, we can find a base $\{e_1,e_2,e_3\}$ such that
    \begin{align*}
      \mu(e_1,e_1)&=e_2,\\
      \mu(e_2,e_1)&=0,\\
      \mu(e_3,e_1)&=0.
    \end{align*}
    The Leibniz identity for $(x,e_1,e_1)$ leads again to $\mu(x,e_2)=0$, whereas using it for $(e_2,e_2,e_3)$ and the nilpotency of $R_{e_3}$ lead to $\mu(e_2,x)=0$. Finally the Leibniz identity for $(e_1,e_1,e_3)$ and $(e_3,e_3,e_3)$ implies that
    \begin{align*}
      \mu(e_1,e_3)&=ae_2,\\
      \mu(e_3,e_3)&=be_2.
    \end{align*}
    If we consider a change of base $\{x_1,x_2,x_3\}$ such that $\mu(x_2,x_1)=0$, $\mu(x_3,x_1)=0$ and $\mu(x_2,x_3)=0$, we remain in this family of nilpotent Leibniz algebras, and the nullity or non nullity of $a$ and $b$ are preserved under this change of basis. Then, if $a\neq0$ considering $x_1=e_1$, $x_2=e_2$ and $x_3=\tfrac{1}{a}e_3$, leads to the family of non isomorphic Leibniz algebras $\mu_{2,b}$ given by
    \begin{align*}
      \mu_{2,b}(e_1,e_1)&=e_2,\\
      \mu_{2,b}(e_3,e_3)&=be_2,\\
      \mu_{2,b}(e_1,e_3)&=e_2.
    \end{align*}
    If $a=0$ but $b\neq0$, we can analogously take $b=1$ leading to the sole algebra $\mu_{3}$ given by
      \begin{align*}
       \mu_3(e_1,e_1)&=e_2,\\
       \mu_3(e_3,e_3)&=e_2.
    \end{align*}
    Finally if $a=b=0$ we obtain the algebra $\mu_4$ given by $\mu_4(e_1,e_1)=e_2$
    \smallskip

    For the $(b)$ case, there exists a basis $\{e_1,e_2,e_3\}$ such that $\mu(e_2,e_1)=e_3$. The nilpotency of $R_{e_3}$, $R_{e_2}$ and the fact that there is no characteristic vector $x$ such that $\mu(x,x)\neq 0$, imply that the Leibniz algebra $\mu$ is in fact a Lie algebra isomorphic to the Heisenberg algebra of dimension 3 i.e.\ $\mu$ is isomorphic to $\mu_5$ given by
    $\mu_5(e_1,e_2)=-\mu_{5}(e_2,e_1)=-e_3$.
  \item If $s(\algl)=(1,1,1)$, it turns out that $\mu$ is the abelian algebra $\mu_6=0$.
\end{enumerate}
The previous analysis shows the following result
\begin{theorem}
  Every nilpotent complex Leibniz algebra is isomorphic to one of the algebras $\mu_i$ with $i=1,3,4,5,6$ or to $\mu_{2,b}$ with $b\in\ncom$.
\end{theorem}

\section{Contractions and perturbations of the Leibniz algebras}
In this section $\varl^n$ will denote the variety of Lie algebras $\lie^n$, or one of the varieties $\leib^n$ or $\leibn^n$. If $\mu_0\in \varl^n$, we denote $\mathcal{O}\!\left(\mu_0\right)$ the orbit of $\mu_0$ under the action of the general linear group $GL(n,\ncom)$ over $\varl^n$:
\[ \begin{array}{rcl}
  GL(n,\ncom)\times \varl^n & \longrightarrow & \varl^n\\
  (f,\mu_0) & \longmapsto & f^{-1}\circ \mu_0\circ (f\times f)
\end{array}\]
where $f^{-1}\circ\mu_0\circ (f\times f)(x,y)=f^{-1}(\mu_0(f(x),f(y)))$.

Let $C$ be an irreducible component of $\varl^n$ containing $\mu_0$, then $\mathcal{O}\!\left(\mu_0\right)\subseteq C$. We can endow naturally the variety $\varl^n$ with two non equivalent topologies: the metric topology induced by the inclusion of $\varl^n$ in $\ncom^{n^3}$, and the Zariski topology. Notice that the latter is contained in the former. As $C$ is closed in the Zariski topology, the adherence $\overline{\mathcal{O}\!\left(\mu_0\right)}^Z$ of the orbit $\mu_0$ is also contained in $C$.

\medskip
In analogy with the Lie algebras we can formally define the notion of limit over the variety $\varl^{n}$ as follows: let $f_{t}\in GL\left( n,\mathbb{C}\right)$ be a family of non-singular endomorphism depending on a continuous parameter $t$, and consider some $\mu\in\varl^{n}$. If for every pair $x,y\in \ncom^n$ the limit
\begin{equation}
  \mu^{\prime}\left(x,y\right):=\lim_{t\rightarrow0}\mu_t\left(x,y\right):=\lim_{t\rightarrow0}\,f_{t}^{-1}%
  \circ\mu\left(  f_{t}\left(  x\right)  ,f_{t}\left(  y\right)\right) \label{GW}
\end{equation}
exists, then $\mu^{\prime}$ is an algebra law of $\varl^n$. We call this new law the \textbf{contraction} of $\mu$ by $\left\{ f_{t}\right\}$. Using the action of $GL\left( n,\mathbb{C}\right)$ over the variety $\varl^n$, it is easy to see that a contraction of $\mu$ corresponds to a point of the closure of the orbit $\mathcal{O}\!\left(\mu\right)$.

It is important to notice that any non trivial contraction $\mu\rightarrow\mu^{\prime}$ satisfies
\begin{align*}
  &\dim\mathcal{O}\!\left(\mu\right)>\dim\mathcal{O}\!\left(\mu ^{\prime}\right),\\
  &\dim Z_R(\mu)\leq\dim Z_R(\mu') &&\text{where }\ Z_R(\mu)=\left\{x\in\ncom^n\,:\,\mu(y,x)=0,\ \forall y\in\ncom^n\right\},\\
  &\,s(\mu)\geq s(\mu') &&\text{in the nilpotent case}.
\end{align*}
Therefore every component $C$ containing $\mu_0$ contains also any of its contractions.

\begin{definition}
  Assuming the non standard analysis (I.S.T.) of Nelson \cite{Ne}, let $\mu_0$ be a standard law of $\varl^n$. A perturbation $\mu$ of $\mu_0$ over $\varl^n$ is another law in $\varl^n$ satisfying the condition $\mu(x,y)\sim\mu_0(x,y)$ for every standard $x,y$ over $\ncom^n$, where $a\sim b$ means that the vector $a-b$ is infinitesimally small.
\end{definition}

In particular if $\mu'=\lim_{t\rightarrow0}\mu_t$ is a contraction of $\mu$, for every $t_0$ infinitesimally small, the law $\mu_{t_0}$ is isomorphic to $\mu$ and is in fact a perturbation of $\mu'$. Such remark encodes perfectly the link between the notions of perturbation and contraction.\medskip

\noindent{\bf Consequence.} The invariants of the nilpotent laws characterizing the irreducibles components are the stable invariants under perturbation. In particular if $\widetilde{\mu}$ is a perturbation of $\mu$, then
\begin{align*}
  &\dim\mathcal{O}\!\left(\widetilde{\mu}\right)>\dim\mathcal{O}\!\left(\mu\right),\\
  &\dim Z_R(\widetilde{\mu})\leq\dim Z_R(\mu),\\
  &\,s(\widetilde{\mu})\geq s(\mu) && \text{in the nilpotent case}.\qquad\qquad\qquad\qquad\qquad\qquad
\end{align*}

\begin{definition}
  A standard law $\mu\in\varl^n$ is \textbf{rigid} over $\varl^n$ if any perturbation of $\mu$ is isomorphic to it.
\end{definition}

This definition translates to the non-standard language the classic notion of rigidity. In fact, if any perturbation of $\mu$ is isomorphic to $\mu$, its halo (i.e.\ the class of laws $\mu'$ such that $\mu'\sim\mu$) is contained on the orbit $\mathcal{O}\!\left(\mu\right)$. This implies that the orbit is open and, by the transfer principle, we obtain the equivalence. In particular we obtain that the rigid algebras cannot be obtained by contraction and that the rigidity of $\mu\in \varl^n$ over $\varl^n$ implies that $\overline{\mathcal{O}\!\left(\mu\right)}^Z$ is an irreducible component of the variety $\varl^n$.

\section{The variety \texorpdfstring{$\bm{\leibn^3}$}{LeibN(3)}}
In this section, using the notions of the previous paragraph, we determine the irreducibles components of the variety $\leibn^3$.

\begin{enumerate}
\item {\it The law $\mu_1$ (sec.\ \ref{seccion classification}) is rigid over $\leibn^3$.} Clear as it is the only nilpotent Leibniz algebra with a maximal characteristic sequence.

\item {\it $\mu_{2,b}$ ($b\neq0$), $\mu_3$ and $\mu_5$ are not contractions of $\mu_1$.} The dimension of the center cannot decrease with a contraction, however $\dim(Z_R(\mu_1))=2$, $\dim(Z_R(\mu_{2,b}))=1$ (in the $b\neq0$ case), $\dim(Z_R(\mu_3))=1$ and $\dim(Z_R(\mu_5))=1$.

\item {\it The only contractions of $\mu_1$ are isomorphic to $\mu_{2,0}$, $\mu_4$ and $\mu_6$.} It is enough to consider the following family of automorphisms of $\ncom^3$
\[\begin{array}{|l}
   f_t(e_1)=te_1 \\
   f_t(e_2)=t^2e_2 \\
   f_t(e_3)=e_3+te_1\\
\end{array}\qquad
\begin{array}{|l}
  g_t(e_1)=te_1 \\
  g_t(e_2)=t^2e_2 \\
  g_t(e_3)=e_3,
\end{array}\qquad
\begin{array}{|l}
  h_t(e_1)=te_1 \\
  h_t(e_2)=te_2 \\
  h_t(e_3)=te_3,
\end{array}
\]
to obtain the contractions of $\mu_1$ into $\mu_{2,0}$, $\mu_4$ and $\mu_6$ respectively.

\item If $b\neq0$ and $\widetilde{\mu}$ is a perturbation of $\mu_{2,b}$, there exists some $b'\in\ncom$ such that $\widetilde{\mu}$ is isomorphic to $\mu_{2,b'}$. This means that the family $\left\{\mu_{2,b}\right\}_{b\neq0}$ is rigid.

\smallskip

In fact, on one hand we have that $\widetilde{\mu}\not\in\mathcal{O}\!\left(\mu_1\right)$ and thus $s(\widetilde{\mu})=(2,1)$. On the other hand, by the transfer property~\cite{Ne} we can assume that $b$, $\mu_{2,b}$ and $\{e_1,e_2,e_3\}$ are standard and therefore
\begin{align*}
 &\widetilde{\mu}(e_1,e_1)\sim\mu_{2,b}(e_1,e_1),\\
 &\widetilde{\mu}(e_1,e_3)\sim\mu_{2,b}(e_1,e_3),\\
 &\widetilde{\mu}(e_3,e_3)\sim\mu_{2,b}(e_3,e_3),
\end{align*}
and the result follows.

\item {\it The algebras $\mu_{2,0}$, $\mu_3$, $\mu_4$, $\mu_5$ and $\mu_6$ can be perturbed over the laws of the family $\left\{\mu_{2,b}\right\}_{b\neq0}$.}
 In order to obtain perturbed algebras isomorphic to the one of the family $\{\mu_{2,b}\}_{b\neq0}$, it is enough to consider the bilinear maps defined by
\begin{align*}
  &\varphi_2(e_3,e_3)=e_2,& &\varphi_3(e_1,e_3)=e_2,\\
  &\varphi_4(e_3,e_3)=\varphi_3(e_1,e_3)=e_2,& &\varphi_5(e_1,e_1)=e_1,
\end{align*}
and the laws of $\leibn^3$ given by $\mu_{2,0}+\varepsilon\varphi_2$, $\mu_i+\varepsilon\varphi_i$ for $i=3,4,5$, where $\varepsilon\sim 0$ is non zero.
\end{enumerate}

Analogously, we can show that the only contraction of $\mu_3$ and $\mu_{2,0}$ are isomorphic to $\mu_4$ and $\mu_6$, and that the only contraction of $\mu_4$ and $\mu_5$ is $\mu_6$. We can summarize all these results in the following diagram, where the arrows represent contractions and hence the rigid elements are those for which no arrow finishes at them
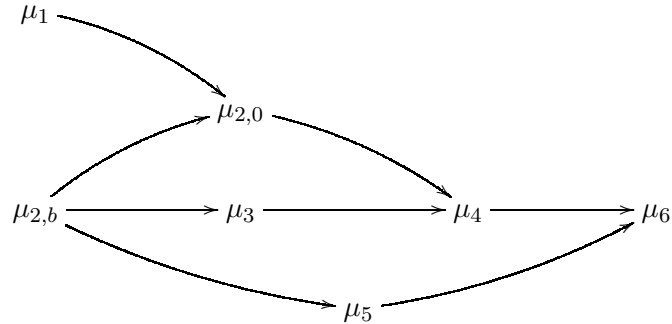
\begin{align*}
    \xymatrix{
         \mu_1     \ar@/^/[rrd]                      &&                        &&              &&        \\
                                                     && \mu_{2,0} \ar@/^/[rrd] &&              &&        \\
         \mu_{2,b} \ar@/^/[rru] \ar[rr] \ar@/_/[rrrd]&& \mu_3     \ar[rr]      && \mu_4 \ar[rr]&& \mu_6  \\
                                                     && &\mu_5     \ar@/_/[rrru]&              &&}
\end{align*}

After this study, we can classify the components of the variety as follows

\begin{theorem}
   The variety $\leibn^3$ is the union of the two irreducibles components $\overline{\mathcal{O}\!\left(\mu_1\right)}^Z$ and $\overline{\bigcup_{b\in\ncom}\mathcal{O}\!\left(\mu_{2,b}\right)}^Z$.
\end{theorem}

\begin{remark}
  In reference~\cite{Alb}, the authors claim that the law $\lambda_5$ of $\leibn^3$ defined (over the basis $\{x_1,x_2,x_3\}$) by
\[\lambda_5(x_2,x_2)=\lambda_5(x_3,x_2)=\lambda_5(x_2,x_3)=x_1,\]
is rigid. Notice however that $\lambda_5$ is isomorphic to $\mu_3$ via the change of basis $e_1=x_2$,$e_2=x_1$ and $e_3=-ix_2+ix_3$. As $\mu_3$ can be perturbed into the family $\{\mu_{2,b}\}$, such claim is not correct.
\end{remark}

\section{The reducibility of the varieties \texorpdfstring{$\bm{\leibn^n}$}{LeibN(n)} and \texorpdfstring{$\bm{\leib^n}$}{Leib(n)}}

Let $\mu_0\in\leibn^n$ be a law with characteristic sequence $s(\mu_0)=(n)$. In that case there exists a basis $\{e_1,\ldots,e_n\}$ of $\ncom^n$ such that $\mu_0(e_i,e_1)=e_{i+1}$ for $i=1,2,\ldots n-1$. Once again, applying the Leibniz identity to $(x,e_1,e_1)$ we obtain that $R_{e_2}=0$. In fact every $R_{e_k}=0$ for $k\geq2$, as by induction $\mu_0(x,e_{k+1})=\mu_0(x,\mu(e_1,e_k))=\mu_0(\mu_0(x,e_1),e_k)-\mu_0(\mu_0(x,e_k),e_1)=0$.

\begin{proposition}
   Any nilpotent Leibniz algebra $\mu$ of dimension $n$ and characteristic sequence $s(\mu)=(n)$ is isomorphic to $\mu_0$, where $\mu_0(e_i,e_1)=e_{i+1}$ for $i=1,\ldots,n-1$.
\end{proposition}

\begin{remark}
The law $\mu$ satisfying $\dim(\mathcal{C}^i(\mu))-\dim(\mathcal{C}^{i+1}(\mu))=1$ for every $i=1,\ldots, n$ is called \textbf{null-filiform} in reference~\cite{Ayu2}.
\end{remark}

As $\mu_0$ is the nilpotent Leibniz algebra with a maximal characteristic sequence, it has to be rigid and $\overline{\mathcal{O}\!\left(\mu_0\right)}^Z$ is an irreducible component of the variety $\leibn^n$. On the other hand if $\mu$ is a non abelian Lie algebra, $\dim(Z(\mu))\leq n-2$ and $\dim(Z_R(\mu_0))=n-1$ implying that $\mu$ cannot be a contraction of $\mu_0$. From this considerations we have the following theorem

\begin{theorem}
  The variety $\leibn^n$ for $n\geq 3$ is reducible.
\end{theorem}

\begin{remark}
$\leibn^2$ is irreducible and the only irreducible component is $\overline{\mathcal{O}\!\left(\mu\right)}^Z$, where $\mu$ is the law defined over the basis $\{e_1,e_2\}$ by $\mu(e_1,e_1)=e_2.$
\end{remark}

Let $\algl=(\ncom^n,\mu)$ be a Leibniz algebra. It is clear that $Z_R(\mu)$ is an ideal of $\algl$ that contains the elements of the form $\mu(x,y)+\mu(y,x)$, $\mu(x,x)$ and $\mu(\mu(x,y),\mu(y,x))$ with $x,y\in\ncom^n$. Thus $\algl/Z_R(\mu)$ is a Lie algebra, which shows the following claim
\begin{center}
{\it Every Leibniz algebra which is not a Lie algebra verifies $Z_R(\mu)\neq0$.}
\end{center}
\begin{theorem}
 A $\lie^n$-rigid Lie algebra without center is also rigid over $\leib^n$.
\end{theorem}
\begin{proof}
Let $\mu$ be a rigid Lie algebra without center. Let $\widetilde{\mu}$ be a perturbation of $\mu$ in $\leibn^n$. As $Z(\mu)=0$ then $Z_R(\widetilde{\mu})=0$ and $\widetilde{\mu}\in\lie^n$. By the rigidity of $\mu$ over $\lie^n$, $\widetilde{\mu}$ is isomorphic to $\mu$.
\end{proof}

\begin{theorem}
   A Lie algebra with non null center cannot be rigid over $\leib^n$.
\end{theorem}
\begin{proof}\mbox{}\\
  Let $\mu$ be a Lie algebra with non null center. We may assume $n$ and $\mu_0$ standard. Let $x$ be a generator of the Lie algebra, $y$ a non zero vector of the center and $\varphi$ the bilinear algebra such that its only non vanishing product is $\varphi(x,x)=y$. Thus the perturbation $\widetilde{\mu}$ of $\mu$ given by $\widetilde{\mu}=\mu+\varepsilon\varphi$ where $\varepsilon\sim0$ is non zero, is a Leibniz algebra that is not a Lie algebra, then $\widetilde{\mu}$ cannot be isomorphic to $\mu$.
\end{proof}

\begin{corollary}
 The variety $\leib^n$ is reducible for $n\geq2$. In fact,
\begin{itemize}
\item $\leib^6$ has at least 5 irreducible components,

\item $\leib^7$ has at least 8 irreducible components,

\item $\leib^8$ has at least 33 irreducible components.

\item $\leib^9$ has at least 41 irreducible components.
\end{itemize}

For $n\geq 81$, the number of irreducible components of $\leib^n$ is lower from below by $\Gamma(\sqrt{n})$, where $\Gamma$ is Euler gamma function (see~\cite{Car} and~\cite{Goz}).
\end{corollary}

\begin{remark}
The variety $\leib^2$ is the union of the two irreducibles components $\overline{\mathcal{O}(\varphi_1)}^Z$ and $\overline{\mathcal{O}(\varphi_2)}^Z$, where the laws are defined, in the basis $\{e_1,e_2\}$ of $\ncom^2$, by $\varphi_1(e_1,e_2)=-\varphi_1(e_2,e_1)=e_2$ (Lie algebra) and $\varphi(e_2,e_1)=e_2$.
\end{remark}

\section*{Acknowledgments}
The first author is supported by the research project MTM2006-09152 of the Ministerio de Ecucaci\'on y Ciencia. This work is a translation from the french version, made by the second author, of the paper \href{http://www.heldermann.de/JLT/JLT17/JLT173/jlt17034.htm}{Sur la R\'eductibilit\'e des Vari\'et\'es des Lois d'Alg\`ebres de Leibniz Complexes}  published at the J.\ Lie Theory  \textbf{17} (2007), No. 3, 617--624.

\end{document}